\documentclass[leqno,11pt]{amsart}
\usepackage{graphicx,pstricks,pst-plot}
\usepackage[utf8]{inputenc}
\usepackage[normalem]{ulem}
\usepackage{epsfig}
\usepackage{mathtools}
\usepackage{float}
\mathtoolsset{showonlyrefs}
\usepackage{amsfonts}
\usepackage{graphicx}
\usepackage{graphicx,color}
\usepackage{footnote}
\usepackage{amsthm}
\usepackage{amsbsy}
\usepackage{amsmath}
\usepackage{amssymb}
\usepackage{rotating}
\usepackage{amsmath}
\usepackage{array}
\usepackage{booktabs}
\usepackage[colorlinks=true,citecolor=blue,linkcolor=blue,pagebackref=true]{hyperref}
\usepackage[T1]{fontenc}
\hypersetup{colorlinks=true,citecolor=red,linkcolor=blue,filecolor=blue,urlcolor=blue}

\usepackage[english]{babel}

\usepackage[
lmargin=3.0cm,
rmargin=3.0cm,
tmargin=4.0cm,
bmargin=4.0cm
]{geometry}

\usepackage{tikz}
\usepackage{pgfplots}
\pgfplotsset{compat=1.18}

\theoremstyle{plain}
\newtheorem{theorem}{\indent\sc Theorem}[section]
\newtheorem{lemma}[theorem]{\indent\sc Lemma}
\newtheorem{corollary}[theorem]{\indent\sc Corollary}
\newtheorem{proposition}[theorem]{\indent\sc Proposition}

\theoremstyle{definition}

\newtheorem{remark}[theorem]{\indent\sc Remark}

\makeatletter
\@namedef{subjclassname@2020}{%
\textup{2020} Mathematics Subject Classification}
\makeatother

\numberwithin{equation}{section}

\title[Heintze--Karcher Inequalities and Serrin-Type Rigidity]
{Heintze--Karcher-Type Inequalities for the $p$-Laplacian and Serrin-Type Rigidity}

\author[M.N. Soares]{Matheus N. Soares$^{1, \ast}$}

\address{$^2$Matheus N. Soares \endgraf 
Escola Politécnica de Pernambuco, \endgraf
Universidade de Pernambuco,
50.720-001, Recife, Pernambuco, Brazil.}

\email{matheus.nsoares@upe.br}

\author[W.F.C. Barboza]{Weiller F.C. Barboza$^{2}$}

\address{$^1$Weiller F. Chaves Barboza \endgraf 
Departamento de Matem\'atica, \endgraf
Universidade Federal de Campina Grande,
58.429-970 Campina Grande, Para\'{\i}ba, Brazil.} 

\email{weiller@mat.ufcg.edu.br}



\thanks{$^\ast$Corresponding author.}

\begin{document}

\begin{abstract}
We establish a parameterized Heintze--Karcher-type inequality for positive solutions of nonlinear Dirichlet problems involving the $p$-Laplacian, with $p\geq2$, on bounded Riemannian domains. Under a Ricci curvature lower bound, positive boundary mean curvature, and suitable structural and approximation assumptions, the estimate follows from a regularized Reilly-type identity and a weighted Hessian inequality. We compare the resulting bound with previous
estimates and investigate the geometry associated with equality. For $p=2$, an exact deficit identity allows the pointwise condition $f'\leq nk$ to be replaced by a weighted integral condition. As applications, we obtain Heintze--Karcher and Soap Bubble-type rigidity results, with equality forcing the domain to be a metric ball and the solution to be radial.
\end{abstract}

\maketitle
\tableofcontents
\vspace{-5em}

\section{Introduction}

Among nonlinear elliptic operators, the $p$-Laplacian plays a central role as a natural generalization of the classical Laplace--Beltrami operator. On a Riemannian manifold $(M^n,g)$, it is defined by
\begin{equation}
\Delta_p u
=
\operatorname{div}\left(|\nabla u|^{p-2}\nabla u\right),
\qquad 1<p<\infty.
\end{equation}
For a bounded domain $\Omega\subset M$, this operator arises in the Euler--Lagrange equation associated with the $p$-energy
\begin{equation}
E_p(u)
=
\frac{1}{p}\int_\Omega |\nabla u|^p\,dV_g,
\qquad
u\in W^{1,p}_0(\Omega),
\end{equation}
where $dV_g$ denotes the Riemannian volume measure and $W^{1,p}_0(\Omega)$ is the closure of $C_c^\infty(\Omega)$ in the $W^{1,p}$-norm. When $p=2$, one recovers the Laplace--Beltrami operator and the usual Dirichlet energy. For $p\neq2$, the operator is nonlinear and is degenerate when $p>2$ or singular when $1<p<2$ at points where $\nabla u=0$. Besides its intrinsic interest in nonlinear potential theory, the $p$-Laplacian appears in models of non-Newtonian fluids, nonlinear elasticity, torsional creep, and nonlinear diffusion; see, for instance, \cite{Lindqvist:2017}.

In this paper, we investigate the nonlinear Dirichlet problem
\begin{equation}\label{eq:intro_problem}
\begin{cases}
\Delta_p u=-f(u) & \text{in }\Omega,\\
u=0             & \text{on }\partial\Omega,
\end{cases}
\end{equation}
where $f$ is a prescribed nonlinearity. Our purpose is to relate analytic properties of solutions of \eqref{eq:intro_problem} to geometric quantities associated with $\Omega$, particularly the mean curvature of its boundary and the Ricci curvature of the ambient manifold. The precise assumptions on the domain, the nonlinearity, and the solutions are stated in the corresponding results.

From a geometric perspective, a central motivation for this study is the Heintze--Karcher inequality. This inequality and its generalizations have become fundamental tools in the study of geometric variational problems, constant mean curvature hypersurfaces, and overdetermined elliptic equations; see \cite{Heintze:78,ros1987compact,borg,julian,qiu,deLima:2026}. Our approach combines a regularized Reilly-type identity with a weighted Hessian inequality to obtain a parameterized Heintze--Karcher inequality. We compare the resulting estimate with previous inequalities and analyze its equality case. In the case $p=2$, we also introduce an error term $\mathcal{E}(u)$ and derive associated Serrin-type and soap-bubble rigidity results. Examples illustrate the role of the structural assumptions imposed on the nonlinearity.

The paper is organized as follows. In Section~2, we introduce the notation and preliminary definitions used throughout the paper. In Section~3, we establish the regularized Reilly-type identity and the weighted Hessian inequality required in the subsequent arguments. Section~4 contains the parameterized Heintze--Karcher inequality, its comparison with previous estimates, and the analysis of necessary conditions for equality. In Section~5, we specialize to the Laplace--Beltrami operator, introduce the error term $\mathcal{E}(u)$, and derive an integral refinement of the Heintze--Karcher inequality together with associated rigidity results, including a Soap Bubble-type theorem. Finally, Section~6 presents an example illustrating that the derivative bound on the nonlinearity need not hold outside the range of the solution.

\section{Preliminaries}

Throughout this paper, $(M^n,g)$ denotes a connected smooth $n$-dimensional Riemannian manifold, with $n\geq2$, and $\Omega\Subset M$ denotes a domain with $C^2$-boundary, unless otherwise stated. Here, $\Omega\Subset M$ means that $\overline{\Omega}$ is compact. All differential operators are understood with respect to the Riemannian metric $g$. We denote by $\nabla$ the Levi-Civita connection of $M$ and also use $\nabla u$ for the gradient of a function $u$. For $u\in C^2(\Omega)$, its Hessian is defined by
\begin{equation}
{\rm Hess}\,u(X,Y)
=
\langle\nabla_X\nabla u,Y\rangle,
\end{equation}
and its Laplace--Beltrami operator is $\Delta u={\rm tr}({\rm Hess}\,u).$ We also regard ${\rm Hess}\,u$ as a self-adjoint endomorphism through the metric identification, so that
\begin{equation}
\langle{\rm Hess}\,u(X),Y\rangle
=
{\rm Hess}\,u(X,Y).
\end{equation}
All inner products and norms are taken with respect to $g$. Following~\cite{ONeill:1983}, the curvature tensor of $M$ is defined by
\begin{equation}
R(X,Y)Z
=
\nabla_{[X,Y]}Z-[\nabla_X,\nabla_Y]Z,
\end{equation}
for all vector fields $X,Y,Z\in\mathfrak{X}(M)$, where $[X,Y]$ denotes the Lie bracket and
$$[\nabla_X,\nabla_Y]=\nabla_X\nabla_Y-\nabla_Y\nabla_X.$$
With this convention, the Ricci tensor is given by
\begin{equation}
{\rm Ric}(X,Y)
=
\sum_{i=1}^{n}\langle R(X,e_i)Y,e_i\rangle,
\end{equation}
where $\{e_i\}_{i=1}^{n}$ is a local orthonormal frame. Along $\partial\Omega$, we denote by $\eta$ the outward unit normal vector field. We adopt the convention
\begin{equation}
\mathcal A_\eta(X)=\nabla_X\eta,
\qquad X\in T\partial\Omega,
\end{equation}
for the shape operator. The associated second fundamental form is $\langle\mathcal A_\eta(X),Y\rangle$, and the non-normalized mean curvature is $H={\rm tr}_{\partial\Omega}\mathcal A_\eta.$ In particular, a Euclidean sphere of radius $r$, oriented by its outward unit normal, has $H=(n-1)/r$. When needed, the normalized mean curvature will be denoted by $\widehat H=H/(n-1)$.

For a function $u\in C^1(\overline{\Omega})$, we denote its outward normal derivative by $u_\eta=\langle\nabla u,\eta\rangle,$ and its tangential gradient along $\partial\Omega$ by
\begin{equation}
\nabla^\partial u=\nabla u-u_\eta\eta.
\end{equation}
We write $\Delta^\partial u$ for the Laplace--Beltrami operator of the restriction of $u$ to $\partial\Omega$. For $u\in C^2(\overline{\Omega})$, these conventions give
\begin{equation}
\Delta u
=
\Delta^\partial u+Hu_\eta+{\rm Hess}\,u(\eta,\eta)
\qquad\text{on }\partial\Omega.
\end{equation}
In particular, if $u=0$ on $\partial\Omega$, then $\nabla^\partial u=0$ and $\nabla u=u_\eta\eta$ there. We denote the Riemannian volume element on $\Omega$ by $d\Omega$ and the induced volume element on $\partial\Omega$ by $d\sigma$. Accordingly,
\begin{equation*}
|\Omega|=\int_\Omega d\Omega
\qquad \text{and} \qquad
|\partial\Omega|=\int_{\partial\Omega}d\sigma.
\end{equation*}

Throughout this paper, we consider a Ricci curvature lower bound of the form
\begin{equation}\label{Ricci}
{\rm Ric}\geq(n-1)kg
\quad\text{on }\Omega,
\end{equation}
where $k\in\mathbb R$. More precisely, condition \eqref{Ricci} means that
\begin{equation}
{\rm Ric}_q(X,X)\geq(n-1)k\,g_q(X,X),
\end{equation}
for every $q\in\Omega$ and every $X\in T_qM$. Restrictions on the sign of $k$ will be specified in the corresponding results. Standard examples include Euclidean space $\mathbb R^n$, the unit sphere $\mathbb S^n$, and hyperbolic space $\mathbb H^n$ of sectional curvature $-1$, for which equality holds with $k=0$, $k=1$, and $k=-1$, respectively. Under these assumptions, we study the nonlinear Dirichlet problem
\begin{equation}\label{P}
\begin{cases}
\Delta_pu=-f(u), & \text{in }\Omega,\\
u=0, & \text{on }\partial\Omega,
\end{cases}
\end{equation}
where $f\in C^1(\mathbb R)$ and
\begin{equation}
\Delta_pu
=
{\rm div}\left(|\nabla u|^{p-2}\nabla u\right).
\end{equation}
Although the $p$-Laplacian is defined for $1<p<\infty$, we restrict throughout this paper to $2\leq p<\infty$. When $p=2$, the operator reduces to the Laplace--Beltrami operator, and \eqref{P} becomes a semilinear elliptic problem. A standard example is the Dirichlet eigenvalue problem for the $p$-Laplacian, corresponding to $f(t)=\lambda|t|^{p-2}t$. This example illustrates the structure of the equation, although it need not satisfy the additional hypotheses imposed in our results.

Unless additional regularity is explicitly assumed, equation \eqref{P} is understood in the weak sense.
Thus, a function $u\in W^{1,p}_0(\Omega)\cap C^{1,\alpha}(\overline{\Omega})$, for some $0<\alpha<1$, is a solution if
\begin{equation}
\int_\Omega
|\nabla u|^{p-2}\langle\nabla u,\nabla\varphi\rangle
d\Omega
=
\int_\Omega f(u)\varphi\,d\Omega
\end{equation}
for every $\varphi\in C_c^\infty(\Omega)$. Here, $W^{1,p}_0(\Omega)$ denotes the closure of $C_c^\infty(\Omega)$ in the $W^{1,p}$-norm. Positivity of the solution and further structural conditions on $f$ will be stated explicitly whenever required.

For $p>2$, the degeneracy of the operator at critical points prevents one from assuming higher classical regularity without further justification. Under suitable hypotheses, $C^{1,\alpha}$ regularity up to the boundary is provided by the regularity theory for degenerate elliptic equations; see \cite{Lieberman:1988}.
The differential identities used below are first established for functions in $C^3(\Omega)\cap C^2(\overline{\Omega})$. For $\varepsilon>0$, we introduce the regularized operator
\begin{equation}\label{eq:unif}
\Delta_{p,\varepsilon}v
:=
{\rm div}\left(
(|\nabla v|^2+\varepsilon)^{\frac{p-2}{2}}\nabla v
\right).
\end{equation}
When applying this operator to a function $u_\varepsilon$, we use the notation
\begin{equation}
W=|\nabla u_\varepsilon|^2+\varepsilon,
\qquad
\Delta_{\infty,\varepsilon}u_\varepsilon
=
\frac{
{\rm Hess}\,u_\varepsilon
(\nabla u_\varepsilon,\nabla u_\varepsilon)
}{W}.
\end{equation}
Consequently,
\begin{equation*}
\Delta_{p,\varepsilon}u_\varepsilon
=
W^{\frac{p-2}{2}}
\left(
\Delta u_\varepsilon
+(p-2)\Delta_{\infty,\varepsilon}u_\varepsilon
\right).
\end{equation*}
For each fixed $\varepsilon>0$, this regularization removes the degeneracy at points where $\nabla u_\varepsilon=0$. The coefficients satisfy uniform ellipticity bounds on sets where the gradient is bounded, with constants that may depend on $\varepsilon$ and the gradient bound. For the unregularized operator, we write
\begin{equation*}
\Omega_{\mathrm{reg}}
=
\{q\in\Omega:|\nabla u(q)|\neq0\}.
\end{equation*}
At points of $\Omega_{\mathrm{reg}}$ where $u$ is twice
differentiable, we define
\begin{equation*}
\Delta_\infty u
=
\frac{{\rm Hess}\,u(\nabla u,\nabla u)}
{|\nabla u|^2}.
\end{equation*}
Thus, $\Delta_\infty$ denotes the normalized infinity-Laplacian, and
\begin{equation*}
\Delta_pu
=
|\nabla u|^{p-2}
\left(\Delta u+(p-2)\Delta_\infty u\right)
\qquad\text{on }\Omega_{\mathrm{reg}}.
\end{equation*}

The regularized identities are first established for sufficiently regular functions. To apply them to a weak solution $u$ of \eqref{P}, we use approximating functions $u_\varepsilon$, vanishing on
$\partial\Omega$, such that
\begin{equation}
\Delta_{p,\varepsilon}u_\varepsilon=-h_\varepsilon,
\qquad
u_\varepsilon\to u
\quad\text{in }C^1(\overline{\Omega}),
\qquad
h_\varepsilon\to f(u)
\quad\text{in }L^2(\Omega).
\end{equation}
Under these approximation assumptions, the passage to the limit is justified by Lemma \ref{lem:weak-reilly-limit}. The  structural assumptions on $f$ are then applied directly to the limiting solution $u$.

\section{A Regularized Reilly Identity and Hessian Estimates}

In this section, we establish a regularized Reilly-type identity and a weighted Hessian inequality that will be used in the proofs of the main results. The corresponding non-regularized identity was considered by dos Santos and Soares in \cite{dosSantos:2023}, with the opposite sign convention for the $p$-Laplacian. For completeness, we give a proof of the regularized identity using the conventions introduced in the previous section.

\begin{proposition}\label{prop:4.1}
Let $M^n$ be a Riemannian manifold and let $\Omega\Subset M$ be a domain with $C^2$-boundary. Let $p\geq2$ and $\varepsilon>0$. Then, for every
$u_\varepsilon\in C^3(\Omega)\cap C^2(\overline{\Omega})$, we have
\begin{equation}\label{eq_3.1.28}
\begin{split}
&\int_\Omega W^{p-2}
\left(
|{\rm Hess}\,u_\varepsilon|^2
+{\rm Ric}(\nabla u_\varepsilon,
          \nabla u_\varepsilon)
+(p-2)^2
(\Delta_{\infty,\varepsilon}u_\varepsilon)^2
\right)d\Omega
\\
&\quad=
\int_\Omega
(\Delta_{p,\varepsilon}u_\varepsilon)^2d\Omega
-2(p-2)\int_\Omega W^{p-3}
|{\rm Hess}\,u_\varepsilon
(\nabla u_\varepsilon)|^2d\Omega
-
\int_{\partial\Omega}
W^{p-2}Q(u_\varepsilon)d\sigma,
\end{split}
\end{equation}
where $W=|\nabla u_\varepsilon|^2+\varepsilon$ and
\begin{equation}\label{eq_3.1.28.mod}
\begin{split}
Q(u_\varepsilon)
&=
(u_\varepsilon)_\eta
\left(
\Delta^\partial u_\varepsilon
+H(u_\varepsilon)_\eta
\right)
\\
&\quad+
\langle
\mathcal A_\eta(\nabla^\partial u_\varepsilon),
\nabla^\partial u_\varepsilon
\rangle
-
\langle
\nabla^\partial u_\varepsilon,
\nabla^\partial((u_\varepsilon)_\eta)
\rangle.
\end{split}
\end{equation}
\end{proposition}

\begin{proof}
Set $X=W^{\frac{p-2}{2}}\nabla u_\varepsilon.$ Thus, ${\rm div}X=\Delta_{p,\varepsilon}u_\varepsilon$. With our curvature convention, the divergence identity
\begin{equation*}
{\rm div}\left(({\rm div}X)X-\nabla_X X\right)
=
({\rm div}X)^2
-
{\rm tr}\left((\nabla X)^2\right)
-
{\rm Ric}(X,X)
\end{equation*}
follows by commuting covariant derivatives. Here,
\begin{equation*}
{\rm tr}\left((\nabla X)^2\right)
=
\sum_{i,j=1}^{n}
\langle\nabla_{e_i}X,e_j\rangle
\langle\nabla_{e_j}X,e_i\rangle
\end{equation*}
for a local orthonormal frame $\{e_i\}_{i=1}^{n}$. Since $\nabla W
= 2\,{\rm Hess}\,u_\varepsilon (\nabla u_\varepsilon),$direct computation gives
\begin{equation}
\begin{split}
{\rm tr}\left((\nabla X)^2\right)
&=
W^{p-2}|{\rm Hess}\,u_\varepsilon|^2
+
2(p-2)W^{p-3}
|{\rm Hess}\,u_\varepsilon
(\nabla u_\varepsilon)|^2
\\
&\quad+
(p-2)^2W^{p-2}
(\Delta_{\infty,\varepsilon}u_\varepsilon)^2.
\end{split}
\end{equation}
Moreover,
\begin{equation*}
{\rm Ric}(X,X)
=
W^{p-2}{\rm Ric}
(\nabla u_\varepsilon,\nabla u_\varepsilon).
\end{equation*}

We next compute the boundary term. Writing $a=W^{\frac{p-2}{2}}$, we have
\begin{equation*}
{\rm div}X
=
a\Delta u_\varepsilon
+
\langle\nabla a,\nabla u_\varepsilon\rangle,
\end{equation*}
and
\begin{equation*}
\nabla_X X
=
a\langle\nabla a,\nabla u_\varepsilon\rangle
\nabla u_\varepsilon
+
a^2\nabla_{\nabla u_\varepsilon}
\nabla u_\varepsilon.
\end{equation*}
Consequently,
\begin{equation*}
\left\langle
({\rm div}X)X-\nabla_X X,\eta
\right\rangle
=
W^{p-2}\left(
(\Delta u_\varepsilon)(u_\varepsilon)_\eta
-
{\rm Hess}\,u_\varepsilon
(\nabla u_\varepsilon,\eta)\right),
\end{equation*}
for every vector field $Y$ tangent to $\partial\Omega$,
\begin{equation*}
{\rm Hess}\,u_\varepsilon(Y,\eta)
=
Y((u_\varepsilon)_\eta)
-
\langle
\nabla^\partial u_\varepsilon,
\mathcal A_\eta(Y)
\rangle.
\end{equation*}
Using $\nabla u_\varepsilon
=
\nabla^\partial u_\varepsilon
+(u_\varepsilon)_\eta\eta,$
and
\begin{equation*}
\Delta u_\varepsilon
=
\Delta^\partial u_\varepsilon
+H(u_\varepsilon)_\eta
+{\rm Hess}\,u_\varepsilon(\eta,\eta),
\end{equation*}
we obtain
\begin{equation*}
\left\langle
({\rm div}X)X-\nabla_X X,\eta
\right\rangle
=
W^{p-2}Q(u_\varepsilon).
\end{equation*}
Integrating the divergence identity yields \eqref{eq_3.1.28}. For the stated regularity, the integration may be performed first on interior parallel domains and then passed to the boundary, the vector field and the expression for its divergence extend continuously to $\overline{\Omega}$.
\end{proof}

\begin{lemma}\label{lem:weighted-hessian}
Under the assumptions of Proposition~\ref{prop:4.1}, the pointwise inequality
\begin{equation}\label{eq:ineqq}
\frac{1}{n}
(\Delta_{p,\varepsilon}u_\varepsilon)^2
\leq
\mathcal B_\varepsilon(u_\varepsilon)
\end{equation}
holds throughout $\Omega$, where 
\begin{equation*}
\begin{split}
\mathcal B_\varepsilon(u_\varepsilon)
&:=
W^{p-2}|{\rm Hess}\,u_\varepsilon|^2 +
2(p-2)W^{p-3}
|{\rm Hess}\,u_\varepsilon
(\nabla u_\varepsilon)|^2
\\
&\quad+
(p-2)^2W^{p-2}
(\Delta_{\infty,\varepsilon}u_\varepsilon)^2.
\end{split}
\end{equation*}
\end{lemma}

\begin{proof}
Fix $q\in\Omega$. If $\nabla u_\varepsilon(q)=0$, then $W(q)=\varepsilon$ and
\begin{equation*}
\Delta_{p,\varepsilon}u_\varepsilon(q)
=
\varepsilon^{\frac{p-2}{2}}
\Delta u_\varepsilon(q).
\end{equation*}
Therefore, the usual trace inequality gives
\begin{equation*}
\frac{1}{n}
(\Delta_{p,\varepsilon}u_\varepsilon)^2
\leq
\varepsilon^{p-2}
|{\rm Hess}\,u_\varepsilon|^2
=
\mathcal B_\varepsilon(u_\varepsilon)
\end{equation*}
at $q$. Suppose now that $\nabla u_\varepsilon(q)\neq0$. Choose an orthonormal frame at $q$ such that
\begin{equation}\label{eq:normal}
e_1
=
\frac{\nabla u_\varepsilon}
{|\nabla u_\varepsilon|}.
\end{equation}
Write
\begin{equation*}
\mathfrak h_{ij}
=
{\rm Hess}\,u_\varepsilon(e_i,e_j),
\qquad
\mathcal S=\sum_{j=2}^{n}\mathfrak h_{jj},
\end{equation*}
and set
\begin{equation*}
\alpha_\varepsilon
=
1+(p-2)\frac{|\nabla u_\varepsilon|^2}{W}.
\end{equation*}
In particular, $\alpha_\varepsilon\geq1$. Expanding the Hessian terms in this frame gives
\begin{equation}\label{eq:123}
\begin{split}
\mathcal B_\varepsilon(u_\varepsilon)
&=
W^{p-2}
\left(
\alpha_\varepsilon^2\mathfrak h_{11}^2
+2\alpha_\varepsilon
\sum_{j=2}^{n}\mathfrak h_{1j}^2
+\sum_{i,j=2}^{n}\mathfrak h_{ij}^2
\right)
\\
&\geq
W^{p-2}
\left(
\alpha_\varepsilon^2\mathfrak h_{11}^2
+\frac{\mathcal S^2}{n-1}
\right).
\end{split}
\end{equation}
On the other hand,
\begin{equation*}
\Delta_{p,\varepsilon}u_\varepsilon
=
W^{\frac{p-2}{2}}
\left(
\alpha_\varepsilon\mathfrak h_{11}
+\mathcal S
\right).
\end{equation*}
By the Cauchy--Schwarz inequality,
\begin{equation}\label{eq:321}
\begin{split}
(\Delta_{p,\varepsilon}u_\varepsilon)^2
&=
W^{p-2}
\left(
\alpha_\varepsilon\mathfrak h_{11}
+\mathcal S
\right)^2
\leq
nW^{p-2}
\left(
\alpha_\varepsilon^2\mathfrak h_{11}^2
+\frac{\mathcal S^2}{n-1}
\right).
\end{split}
\end{equation}
Combining \eqref{eq:123} and \eqref{eq:321} proves \eqref{eq:ineqq}.
\end{proof}

Observe that, at a point where $\nabla u_\varepsilon\neq0$, equality in \eqref{eq:ineqq} holds if and only if
\begin{equation}
\mathfrak h_{1j}=0
\quad\text{for }j=2,\ldots,n,
\qquad
\mathfrak h_{ij}
=
\alpha_\varepsilon\mathfrak h_{11}\delta_{ij}
\quad\text{for }i,j=2,\ldots,n.
\end{equation}
At a critical point, equality holds if and only if
\begin{equation*}
{\rm Hess}\,u_\varepsilon
=
\frac{\Delta u_\varepsilon}{n}g.
\end{equation*}
For the corresponding non-regularized pointwise inequality on $\Omega_{\mathrm{reg}}$,
$\alpha_\varepsilon$ is replaced by $p-1$. These are the algebraic equality conditions; their use for limits of regularized solutions requires a separate justification.
If $u_\varepsilon=0$ on $\partial\Omega$, then $\nabla^\partial u_\varepsilon=0$ and $Q(u_\varepsilon)=H(u_\varepsilon)_\eta^2.$ Consequently, Proposition~\ref{prop:4.1} and Lemma~\ref{lem:weighted-hessian} imply
\begin{equation}\label{eq:regularized-reilly-inequality}
\begin{split}
&\int_\Omega
W^{p-2}
{\rm Ric}(\nabla u_\varepsilon,\nabla u_\varepsilon)
d\Omega
-
\frac{n-1}{n}
\int_\Omega
(\Delta_{p,\varepsilon}u_\varepsilon)^2d\Omega
\\
&\qquad\leq
-
\int_{\partial\Omega}
W^{p-2}H(u_\varepsilon)_\eta^2d\sigma.
\end{split}
\end{equation}

We first pass to the non-regularized inequality for a fixed sufficiently regular function.

\begin{proposition}\label{prop}
Let $M^n$ be a Riemannian manifold and let
$\Omega\Subset M$ be a domain with $C^2$-boundary.
Let $p\geq2$ and
$u\in C^3(\Omega)\cap C^2(\overline{\Omega})$,
with $u=0$ on $\partial\Omega$. Then
\begin{equation*}
\begin{split}
&\int_\Omega
|\nabla u|^{2p-4}
{\rm Ric}(\nabla u,\nabla u)d\Omega
-
\frac{n-1}{n}
\int_\Omega(\Delta_pu)^2d\Omega
\\
&\qquad\leq
-
\int_{\partial\Omega}
|\nabla u|^{2p-4}Hu_\eta^2d\sigma.
\end{split}
\end{equation*}
For $p=2$, the factor $|\nabla u|^{2p-4}$ is understood to be identically equal to $1$.
\end{proposition}

\begin{proof}
Apply \eqref{eq:regularized-reilly-inequality} with $u_\varepsilon=u$ for every $\varepsilon>0$. Since $u\in C^2(\overline{\Omega})$ and $\overline{\Omega}$ is compact, both $\nabla u$ and ${\rm Hess}\,u$ are bounded. At points where $\nabla u\neq0$, $\Delta_{p,\varepsilon}u\to\Delta_pu$ as $\varepsilon\to0$. At a critical point,
\begin{equation}
\Delta_{p,\varepsilon}u
=
\varepsilon^{\frac{p-2}{2}}\Delta u.
\end{equation}
Hence the same convergence holds there: for $p>2$, $\Delta_pu=0$ at a critical point
of a $C^2$ function, while for $p=2$ the regularized operator is simply $\Delta$. Moreover,
\begin{equation}
|\Delta_{p,\varepsilon}u|
\leq
W^{\frac{p-2}{2}}
\left(
|\Delta u|+(p-2)|{\rm Hess}\,u|
\right),
\qquad W=|\nabla u|^2+\varepsilon.
\end{equation}
For $0<\varepsilon\leq1$, this gives a uniform bound on $\overline{\Omega}$. The curvature and boundary terms are also dominated by integrable functions independent of $\varepsilon$. The desired inequality follows by dominated convergence.
\end{proof}

The preceding proposition concerns a fixed regular function. To apply its conclusion to a weak solution of \eqref{P}, one must additionally justify an approximation compatible with the equation and the boundary terms. The following lemma records a sufficient condition.

\begin{lemma}\label{lem:weak-reilly-limit}
Let $M^n$ be a Riemannian manifold, let $\Omega\Subset M$ be a domain with $C^2$-boundary,
and let $p\geq2$. Suppose that $u\in C^1(\overline{\Omega})$ satisfies $u=0$ on $\partial\Omega$. Assume that there exist $\varepsilon_j>0$, with $\varepsilon_j\to0$, and functions $u_j\in C^3(\Omega)\cap C^2(\overline{\Omega})$, with $u_j=0$ on $\partial\Omega$, such that
\begin{equation*}
u_j\longrightarrow u
\quad\text{in }C^1(\overline{\Omega})
\end{equation*}
and
\begin{equation*}
\Delta_{p,\varepsilon_j}u_j
\longrightarrow -f(u)
\quad\text{in }L^2(\Omega).
\end{equation*}
Then
\begin{equation}\label{eq:weak-reilly-inequality}
\int_\Omega
|\nabla u|^{2p-4}
{\rm Ric}(\nabla u,\nabla u)d\Omega
-
\frac{n-1}{n}\int_\Omega f(u)^2d\Omega
\leq
-
\int_{\partial\Omega}
|\nabla u|^{2p-4}Hu_\eta^2d\sigma.
\end{equation}
\end{lemma}

\begin{proof}
Apply \eqref{eq:regularized-reilly-inequality} to $u_j$, with $W_j=|\nabla u_j|^2+\varepsilon_j$. The convergence in $C^1(\overline{\Omega})$ implies convergence of the curvature and boundary integrals. The assumed strong convergence in $L^2(\Omega)$ gives
\begin{equation}
\int_\Omega
(\Delta_{p,\varepsilon_j}u_j)^2d\Omega
\longrightarrow
\int_\Omega f(u)^2d\Omega.
\end{equation}
Passing to the limit proves
\eqref{eq:weak-reilly-inequality}.
\end{proof}

\section{A Heintze--Karcher-Type Inequality for the \texorpdfstring{$p$}{p}-Laplacian}

In this section, we establish a parameterized Heintze--Karcher-type inequality for positive solutions of \eqref{P}. The proof combines the integral inequality from the previous section with a structural assumption on the nonlinearity. We also describe necessary conditions for
equality and compare the resulting estimate with previous inequalities. Throughout this section, $H$ denotes the non-normalized mean curvature with respect to the outward unit normal. 

{\begin{theorem}\label{thm:main}
Let $M^n$ be a Riemannian manifold and let $\Omega\Subset M$ be a domain with $C^2$-boundary. Suppose \eqref{Ricci} and that $H>0$ on $\partial\Omega$. Let $p\geq2$ and let $u\in W^{1,p}_0(\Omega)\cap
C^{1,\alpha}(\overline{\Omega})$ be a positive weak solution of \eqref{P}. Assume that $f\in C^1(\mathbb R)$, $f(0)>0$, and that there exists $\ell\in\mathbb R$ such that $f(t)\geq\ell$ for $0\leq t\leq\max_{\overline{\Omega}}u.$ Consider, on the same interval, that
\begin{equation}\label{eq:structural-condition}
\begin{cases}
f'(t)\leq nk,
& p=2,\\[2mm]
f'(t)\leq
(nk)^{\frac{p}{2p-2}}
(f(t)-\ell)^{\frac{p-2}{p-1}},
& p>2.
\end{cases}
\end{equation}
Then
\begin{equation}\label{eq:ineq1}
\frac{n-1}{n}
\int_{\partial\Omega}\frac{1}{H}\,d\sigma
\geq
\frac{
\left(\displaystyle\int_\Omega f(u)d\Omega\right)^2
}{
\Gamma_\ell\displaystyle\int_\Omega f(u)d\Omega
+
\frac{p-2}{p}\ell^2|\Omega|
}.
\end{equation}
where 
\begin{equation}
\Gamma_{\ell}=\frac{2p-2}{p}f(0)-\frac{2(p-2)}{p}\ell.
\end{equation}
In particular, when $\ell=0$,
\begin{equation}
\frac{n-1}{n}\frac{2p-2}{p}f(0)
\int_{\partial\Omega}\frac{1}{H}\,d\sigma
\geq
\int_\Omega f(u)d\Omega.
\end{equation}
Moreover, if the equality holds, then:
\begin{enumerate}
\item For $p=2$, $u$ is radial function and $\Omega$ is a metric ball;
\item For $p> 2$ we have two cases:
\subitem If $k>0$, the equality not occur;
\subitem If $k=0$, then $f(u)=f(0)=\ell$ and $u$ is radial function and $\Omega$ is a metric ball.
\end{enumerate}
\end{theorem}}
\begin{proof}
For simplicity, write
\begin{equation}
F=\int_\Omega f(u)d\Omega,
\qquad
J=\int_\Omega f(u)^2d\Omega,
\qquad
G=\int_\Omega|\nabla u|^{2p-2}d\Omega.
\end{equation}
Since $u>0$ in $\Omega$, $u=0$ on $\partial\Omega$,
and $f(0)>0$, the boundary point lemma gives
$u_\eta<0$ on $\partial\Omega$. Indeed, $f(u)>0$
in a sufficiently small boundary neighborhood.
Thus,
\begin{equation}
\nabla u=u_\eta\eta,
\qquad
u_\eta=-|\nabla u|
\quad\text{on }\partial\Omega.
\end{equation}

The weak equation and the divergence theorem give
\begin{equation}\label{eq:flux}
F
=
-\int_{\partial\Omega}
|\nabla u|^{p-2}u_\eta\,d\sigma
=
\int_{\partial\Omega}|u_\eta|^{p-1}d\sigma
>0.
\end{equation}
Moreover, using $f(u)-f(0)$ as a test function,
we obtain
\begin{equation}\label{eq:TERM}
\int_\Omega f(u)^2d\Omega
=
f(0)\int_\Omega f(u)d\Omega
+
\int_\Omega f'(u)|\nabla u|^p d\Omega.
\end{equation}
The test function belongs to $W^{1,p}_0(\Omega)$
because $u$ is bounded and $f\in C^1(\mathbb R)$. Suppose first that $p>2$, and set
\begin{equation*}
a=\frac{p-2}{2p-2},
\qquad
b=\frac{p}{2p-2}.
\end{equation*}
Then $a,b>0$ and $a+b=1$.
By \eqref{eq:structural-condition} and
Hölder's inequality,
\begin{equation}\label{eq:H1}
\begin{split}
\int_\Omega f'(u)|\nabla u|^p d\Omega
&\leq
(nk)^b
\int_\Omega
(f(u)-\ell)^{2a}|\nabla u|^p d\Omega
\\
&\leq
\left(
\int_\Omega(f(u)-\ell)^2d\Omega
\right)^a
(nkG)^b.
\end{split}
\end{equation}
The weighted Young inequality gives
\begin{equation}\label{eq:H1.1}
\int_\Omega f'(u)|\nabla u|^p d\Omega
\leq
a\int_\Omega(f(u)-\ell)^2d\Omega
+bnkG.
\end{equation}
Inserting this estimate into \eqref{eq:TERM},
expanding the square, and using $1-a=b$, we find
\begin{equation*}
bJ
\leq
\left(f(0)-2a\ell\right)F
+a\ell^2|\Omega|
+bnkG.
\end{equation*}
Hence,
\begin{equation}\label{eq:source-estimate}
J-nkG
\leq
\Gamma_\ell F
+\frac{p-2}{p}\ell^2|\Omega|.
\end{equation}

If $p=2$, equation \eqref{eq:TERM} and the
assumption $f'\leq nk$ directly yield
\begin{equation*}
J-nkG\leq f(0)F.
\end{equation*}
This is precisely \eqref{eq:source-estimate}, since $\Gamma_\ell=f(0)$ when $p=2$. Define
\begin{equation}
D_\ell
=
\Gamma_\ell F
+\frac{p-2}{p}\ell^2|\Omega|
\qquad \text{and} \qquad
B=
\int_{\partial\Omega}
H|u_\eta|^{2p-2}d\sigma.
\end{equation}
By Lemma~\ref{lem:weak-reilly-limit},
\begin{equation}
\begin{split}
B
&\leq
\frac{n-1}{n}J
-
\int_\Omega|\nabla u|^{2p-4}
{\rm Ric}(\nabla u,\nabla u)d\Omega
\\
&\leq
\frac{n-1}{n}(J-nkG)
\leq
\frac{n-1}{n}D_\ell.
\end{split}
\end{equation}
Since $\ell\leq f(0)$, we have $\Gamma_\ell\geq\frac{2}{p}f(0)>0.$ Consequently, $D_\ell>0$. Finally, \eqref{eq:flux} and the Cauchy--Schwarz inequality imply
\begin{equation}
\begin{split}
F^2
&=
\left(
\int_{\partial\Omega}
\frac{1}{\sqrt H}
\sqrt H\,|u_\eta|^{p-1}d\sigma
\right)^2
\leq
\left(
\int_{\partial\Omega}\frac1H\,d\sigma
\right)B
\\
&\leq
\frac{n-1}{n}D_\ell
\int_{\partial\Omega}\frac1H\,d\sigma.
\end{split}
\end{equation}
Division by $D_\ell$ proves \eqref{eq:ineq1}.

{For the equality case, if $p>2$ and $k=0$, then equations \eqref{eq:H1}-\eqref{eq:H1.1} and the hypothesis of $f'\leq 0$ implies that
\begin{equation}
0 \geq \int_{\Omega} f'(u)|\nabla u|^{p} d\Omega =  \left(\frac{p-2}{2p-2}\right)\int_{\Omega} (f-\ell)^2 d\Omega \geq 0.
\end{equation}
Hence, $f(u)=\ell.$ By the continuity up the boundary, $\ell=f(0).$ Hence, equation $\eqref{P}$ become the following $p$-torsion problem\footnote{This example also show the importance of $\ell$ in the theorem. If we choose $\ell=0$ in the same setting, the theorem give a weaker estimate with an additional $(2p-2)/p$ term.}. For $p=2$ the argument is the same as done in \cite{deLima:2026}.
\begin{equation}\label{eq:p-torsion}
\begin{cases}
\Delta_pu=-f(0), & \text{in }\Omega,\\
u=0, & \text{on }\partial\Omega,
\end{cases}
\end{equation}
Moreover, equation $\eqref{P}$ then gives
\begin{equation}\label{eq:p-Heintze-Karcher}
\left(\frac{n-1}{n}\right) \int_{\partial\Omega}\frac{1}{H}\,d\sigma = |\Omega|,
\end{equation}
which reduces to the classical Heintze--Karcher equality, according to our convention for the mean curvature. For $k>0$, observe that the equality in Hölder inequality \eqref{eq:H1} implies that exists a number $c\in \mathbb{R}$ such that
\begin{equation}\label{eq:f(u)}
(f(u)-\ell)^2=c|\nabla u|^{2p-2}.
\end{equation}
Moreover, equality in Young inequality implies that
\begin{equation}
\int_{\Omega} (f(u)-\ell)^2 d\Omega = nk \int_{\Omega} |\nabla u|^{2p-2} d\Omega.
\end{equation}
Hence, $c=nk.$ Taking the derivative of equation \eqref{eq:f(u)} with respect to $\nabla u$, we obtain that,
\begin{equation}
f'(u)|\nabla u|^2 = \sqrt{nk}(p-1)|\nabla u|^{p-3}\langle {\rm Hess}\,u(\nabla u) , \nabla u\rangle
\end{equation}
Moreover, equality in equation \eqref{eq:123} implies that
\begin{equation}
\frac{1}{n}\Delta_p u = (p-1)|\nabla u|^{p-2} \Delta_{\infty} u.
\end{equation}
Hence,
\begin{equation}
f'(u)=-\sqrt{\frac{k}{n}}\frac{f(u)}{|\nabla u|} 
\end{equation}
Hence, since $f(0)>0$, by the continuity of $f$ on an neighborhood of $0$, we have that $f'(u)<0.$ On the other hand, by hypothesis of the control in the derivative of $f$, $$f'(u)=(nk)^{\frac{p}{2p-2}}|f(u)-\ell|^{\frac{p-2}{p-1}} > 0,$$ which is an contradiction.}
\end{proof}

\begin{remark}
The exponent in \eqref{eq:structural-condition} is compatible with the homogeneity of the $p$-Laplacian eigenvalue nonlinearity. Indeed, for
\begin{equation}
f(s)=\lambda|s|^{p-2}s,
\qquad \lambda>0,
\end{equation}
we have, for $s\neq0$,
\begin{equation}
f'(s)
=
(p-1)\lambda^{\frac1{p-1}}
|f(s)|^{\frac{p-2}{p-1}}.
\end{equation}
However, $f(0)=0$ in this example, so Theorem~\ref{thm:main} does not apply.
In particular, the theorem is not an eigenvalue estimate.
\end{remark}

\begin{remark}
When $p=2$, all terms involving $\ell$ disappear, the structural condition becomes $f'\leq nk$, and \eqref{eq:ineq1} reduces to
\begin{equation}
\frac{n-1}{n}f(0)
\int_{\partial\Omega}\frac1H\,d\sigma
\geq
\int_\Omega f(u)d\Omega.
\end{equation}
This recovers the form of the corresponding estimate in \cite{deLima:2026}, subject here to the approximation hypotheses stated above.

To compare with \cite{Huang:2025}, consider, on the nonnegative range of $u$,
\begin{equation}
f(t)=1+\lambda t^{p-1},
\qquad
\ell=1,
\qquad
k=K>0,
\end{equation}
where
\begin{equation}
\lambda
=
\frac{(nK)^{p/2}}{(p-1)^{p-1}}.
\end{equation}
For this choice, the structural condition is satisfied with equality. Set
\begin{equation}
A=\lambda\int_\Omega u^{p-1}d\Omega.
\end{equation}
Then
\begin{equation}
F=|\Omega|+A,
\qquad
\Gamma_1=\frac2p,
\qquad
D_1=|\Omega|+\frac2pA.
\end{equation}
Since Huang, Luo and Song use normalized mean curvature, write
$\widehat H=H/(n-1)$. Our estimate becomes
\begin{equation}
\int_{\partial\Omega}
\frac{1}{n\widehat H}\,d\sigma
\geq
\frac{(|\Omega|+A)^2}
{|\Omega|+\frac2pA}.
\end{equation}
On the other hand,
\begin{equation}
\begin{split}
\frac{(|\Omega|+A)^2}
{|\Omega|+\frac2pA}
-
\left(
|\Omega|+\frac{2(p-1)}pA
\right)
=
\frac{(p-2)^2A^2}
{p^2\left(|\Omega|+\frac2pA\right)}
\geq0.
\end{split}
\end{equation}
Therefore, for solutions satisfying the hypotheses of Theorem~\ref{thm:main}, our inequality implies the estimate of Huang, Luo and Song \cite{Huang:2025}. The two right-hand sides coincide for $p=2$,
whereas ours is strictly larger for $p>2$ and $A>0$.
\end{remark}

We conclude with an integral estimate for the deviation of the boundary mean curvature from a constant determined by $f$.

\begin{theorem}\label{thm:boundary-deficit}
Assume the hypotheses of Theorem~\ref{thm:main} with $\ell=0$,
excluding the additional hypotheses used only for the equality conclusions.
Define
\begin{equation*}
m
=
-\frac1{|\partial\Omega|}
\int_{\partial\Omega}
|u_\eta|^{p-2}u_\eta\,d\sigma
=
\frac1{|\partial\Omega|}
\int_\Omega f(u)d\Omega
>0
\end{equation*}
and
\begin{equation*}
H_0
=
\frac{n-1}{n}\frac{2p-2}{p}
\frac{f(0)}{m}.
\end{equation*}
Then
\begin{equation*}
\int_{\partial\Omega}
(H_0-H)|u_\eta|^{2p-2}d\sigma
\geq0.
\end{equation*}
\end{theorem}

\begin{proof}
Set $v=|u_\eta|^{p-1}$ and $c=\frac{n-1}{n}\frac{2p-2}{p}f(0).$ Then
\begin{equation*}
\int_{\partial\Omega}v\,d\sigma
=
m|\partial\Omega|,
\qquad
H_0=\frac{c}{m}.
\end{equation*}
The boundary estimate in the proof of Theorem~\ref{thm:main}, with $\ell=0$, gives
\begin{equation}
\int_{\partial\Omega}Hv^2d\sigma
\leq
c\int_{\partial\Omega}v\,d\sigma
=
cm|\partial\Omega|.
\end{equation}
Consequently,
\begin{equation}
\begin{split}
\int_{\partial\Omega}(H_0-H)v^2d\sigma
&\geq
H_0\int_{\partial\Omega}v^2d\sigma
-
cm|\partial\Omega|
\\
&=
H_0\left(
\int_{\partial\Omega}v^2d\sigma
-
m^2|\partial\Omega|
\right)
\\
&=
H_0\int_{\partial\Omega}(v-m)^2d\sigma
\geq0.
\end{split}
\end{equation}
This proves the assertion.
\end{proof}

\section{An integral refinement of the Heintze--Karcher inequality}

In this section, we restrict to the Laplace--Beltrami case of \eqref{P}, namely
\begin{equation}
\begin{cases}
\Delta u=-f(u), & \text{in }\Omega,\\
u=0, & \text{on }\partial\Omega.
\end{cases}
\end{equation}
We establish an integral refinement of the Heintze--Karcher-type inequality in \cite{deLima:2026}. In particular, the pointwise condition $f'\leq nk$ can be replaced, for the corresponding inequality, by a weighted integral condition along the solution. Throughout this section, $\Omega\Subset M$ is a domain with $C^2$-boundary, $f\in C^1(\mathbb R)$, and $u\in C^3(\Omega)\cap C^2(\overline{\Omega})$ is a positive solution of \eqref{P}, with $p=2$. We retain the conventions that $\eta$ is the outward unit normal and $H$ is the non-normalized mean curvature. The constant $k$ is allowed to belong to $\mathbb R$. We define the signed integral error associated with $u$ by
\begin{equation}\label{eq:error-definition}
    \mathcal E(u)
    =
    \int_\Omega
    \big(f'(u)-nk\big)|\nabla u|^2\,d\Omega.
\end{equation}
The condition
\begin{equation}\label{eta}
\mathcal E(u)\leq0
\end{equation}
is satisfied whenever
\begin{equation}\label{eq:fprime_range}
f'(s)\leq nk
\qquad
\text{for every }
s\in[0,\max_{\overline{\Omega}}u].
\end{equation}
Unlike the pointwise condition, \eqref{eta} allows compensation between the positive and negative parts of $f'(u)-nk$. We impose \eqref{eta} only when explicitly stated. We denote the trace-free Hessian by
\begin{equation}\label{eq:HESS}
\mathring{{\rm Hess}}\,u
=
{\rm Hess}\,u-\frac{\Delta u}{n}g.
\end{equation}
Thus,
\begin{equation*}
|{\rm Hess}\,u|^2
=
|\mathring{{\rm Hess}}\,u|^2
+\frac{(\Delta u)^2}{n}.
\end{equation*}

\begin{lemma}\label{lemma1}
Under the assumptions above, the following
identity holds:
\begin{equation}\label{eq:key-identity}
\begin{split}
&\int_\Omega
|\mathring{{\rm Hess}}\,u|^2d\Omega
+
\int_\Omega
\big[{\rm Ric}-(n-1)kg\big]
(\nabla u,\nabla u)d\Omega
\\
&\qquad=
-\frac1n\int_{\partial\Omega}
u_\eta
\big((n-1)f(0)+nHu_\eta\big)d\sigma
+
\frac{n-1}{n}\mathcal E(u).
\end{split}
\end{equation}
\end{lemma}

\begin{proof}
By \eqref{eq:TERM}, with $p=2$, and the definition
of $\mathcal E(u)$,
\begin{equation}
    \int_\Omega f(u)^2\,d\Omega
    =
    -f(0)\int_{\partial\Omega}u_\eta\,d\sigma
    +nk\int_\Omega|\nabla u|^2\,d\Omega
    +\mathcal E(u).
\end{equation}
Substituting this identity into the classical Reilly formula
and using \eqref{eq:HESS} yields \eqref{eq:key-identity}.
\end{proof}

The next identity expresses the deficit in the Heintze--Karcher-type inequality as a sum of interior and boundary terms.

\begin{lemma}\label{lemma2}
Suppose that $H>0$ on $\partial\Omega$. Set $c=(n-1)/n$. Then
\begin{equation}\label{eq:HK-deficit-identity}
\begin{split}
&c\left(
cf(0)^2\int_{\partial\Omega}\frac1H\,d\sigma
-
f(0)\int_\Omega f(u)d\Omega
+
\mathcal E(u)
\right)
\\
&\quad=
\int_\Omega
|\mathring{{\rm Hess}}\,u|^2d\Omega
+
\int_\Omega
\big[{\rm Ric}-(n-1)kg\big]
(\nabla u,\nabla u)d\Omega
+
\int_{\partial\Omega}
\frac{\big(Hu_\eta+cf(0)\big)^2}{H}\,d\sigma.
\end{split}
\end{equation}
In particular, if \eqref{Ricci} holds, then
\begin{equation}\label{eq:thm1}
\frac{n-1}{n}f(0)^2
\int_{\partial\Omega}\frac1H\,d\sigma
\geq
f(0)\int_\Omega f(u)d\Omega
-
\mathcal E(u).
\end{equation}
No sign assumption on $\mathcal E(u)$ is needed
for this inequality.
\end{lemma}

\begin{proof}
The divergence theorem gives
\begin{equation}
\int_\Omega f(u)d\Omega
=
-\int_{\partial\Omega}u_\eta\,d\sigma.
\end{equation}
Moreover,
\begin{equation}
\begin{split}
\int_{\partial\Omega}
\frac{\big(Hu_\eta+cf(0)\big)^2}{H}\,d\sigma
&=
\int_{\partial\Omega}Hu_\eta^2d\sigma
+
2cf(0)\int_{\partial\Omega}u_\eta\,d\sigma
+
c^2f(0)^2
\int_{\partial\Omega}\frac1H\,d\sigma.
\end{split}
\end{equation}
Adding this equality to \eqref{eq:key-identity} yields \eqref{eq:HK-deficit-identity}. Under \eqref{Ricci}, all terms on the right-hand side of
\eqref{eq:HK-deficit-identity} are nonnegative. Dividing by $c>0$ proves \eqref{eq:thm1}.
\end{proof}

\begin{corollary}\label{cor:HK-integral-condition}
Assume \eqref{Ricci}, $H>0$ on $\partial\Omega$, and $\mathcal E(u)\leq0$. Then
\begin{equation}
\frac{n-1}{n}f(0)^2
\int_{\partial\Omega}\frac1H\,d\sigma
\geq
f(0)\int_\Omega f(u)d\Omega.
\end{equation}
If $\mathcal E(u)<0$, the preceding inequality is strict.
\end{corollary}

\begin{proof}
Both conclusions follow immediately from
\eqref{eq:thm1}.
\end{proof}

\begin{remark}\label{remark:HK_structure}
Corollary~\ref{cor:HK-integral-condition} recovers the Heintze--Karcher-type estimate in \cite[Theorem~1]{deLima:2026} under the integral
condition $\mathcal E(u)\leq0$, instead of the pointwise condition $f'\leq nk$. When $\mathcal E(u)<0$, inequality \eqref{eq:thm1} includes the additional positive correction $-\mathcal E(u)$. The quantity $\mathcal E(u)$ is a signed weighted deviation from the relation $f'=nk$. It is not a norm or a distance: its vanishing alone does not imply $f'(u)=nk$, since cancellation may occur in the integral. The equality analysis below uses the full identity \eqref{eq:HK-deficit-identity}.
\end{remark}

\begin{corollary}\label{cor:HK_equality}
Assume \eqref{Ricci} and $H>0$ on $\partial\Omega$. Equality in \eqref{eq:thm1} holds if and only if
\begin{equation}\label{eq:tracefree-zero}
\mathring{{\rm Hess}}\,u=0
\qquad\text{in }\Omega,
\end{equation}
\begin{equation}\label{eq:ricci-equality-p2}
\big[{\rm Ric}-(n-1)kg\big]
(\nabla u,\nabla u)=0
\qquad\text{in }\Omega,
\end{equation}
and
\begin{equation}\label{eq:boundary_relation}
(n-1)f(0)+nHu_\eta=0
\qquad\text{on }\partial\Omega.
\end{equation}
\end{corollary}

\begin{proof}
Under the stated assumptions, the three integrals on the right-hand side of
\eqref{eq:HK-deficit-identity} have nonnegative continuous integrands.
Their sum vanishes if and only if each integrand vanishes. This is equivalent to \eqref{eq:tracefree-zero}, \eqref{eq:ricci-equality-p2}, and
\eqref{eq:boundary_relation}.
\end{proof}

We now show that equality determines the nonlinearity along the solution and yields a rigidity conclusion. No pointwise assumption on $f'$ is required.

\begin{proposition}\label{prop:rigidity_equality}
Assume \eqref{Ricci}, $H>0$ on $\partial\Omega$, and $f(0)>0$. If equality holds in \eqref{eq:thm1}, then
\begin{equation}\label{eq:obata_equation}
{\rm Hess}\,u
=
-\left(\frac{f(0)}n+ku\right)g
\qquad\text{in }\Omega,
\end{equation}
and
\begin{equation}\label{eq:boundary_obata}
u_\eta
=
-\frac{(n-1)f(0)}{nH}
\qquad\text{on }\partial\Omega.
\end{equation}
Moreover,
\begin{equation}\label{eq:f_linear}
f(s)=f(0)+nks
\qquad
\text{for every }
s\in[0,\max_{\overline{\Omega}}u],
\end{equation}
and $\mathcal E(u)=0$. Consequently, $\Omega$ is a metric ball and $u$ is radial with respect to its center.
\end{proposition}

\begin{proof}
By Corollary~\ref{cor:HK_equality},
\begin{equation}
{\rm Hess}\,u
=
\frac{\Delta u}{n}g
=
-\frac{f(u)}n g.
\end{equation}
The boundary relation \eqref{eq:boundary_obata} follows directly
from \eqref{eq:boundary_relation}. Taking the divergence of the Hessian identity and using
\begin{equation*}
{\rm div}({\rm Hess}\,u)
=
d(\Delta u)+{\rm Ric}(\nabla u,\cdot),
\end{equation*}
we obtain
\begin{equation*}
-\frac1n f'(u)\,du
=
-f'(u)\,du
+
{\rm Ric}(\nabla u,\cdot).
\end{equation*}
Hence,
\begin{equation*}
{\rm Ric}(\nabla u,\cdot)
=
\frac{n-1}{n}f'(u)\,du.
\end{equation*}
Evaluating this identity on $\nabla u$ and using \eqref{eq:ricci-equality-p2}, we find
\begin{equation*}
\frac{n-1}{n}
\big(f'(u)-nk\big)|\nabla u|^2
=
0
\qquad\text{in }\Omega.
\end{equation*}
Therefore, $\big(f'(u)-nk\big)\nabla u=0$ in $\Omega,$
including at the critical points of $u$. Equivalently, $\nabla\big(f(u)-nku\big)=0.$ Since $\Omega$ is connected, $f(u)-nku$ is constant. By continuity up to the boundary and the condition $u=0$ there, this constant is $f(0)$. Thus, $f(u)=f(0)+nku$ in $\overline{\Omega}.$ Since $u$ is positive in $\Omega$ and vanishes on $\partial\Omega$, its range on $\overline{\Omega}$ is $[0,\max_{\overline{\Omega}}u]$.
This proves \eqref{eq:f_linear}, and substitution into the Hessian identity gives \eqref{eq:obata_equation}. The preceding argument also yields
\begin{equation}\label{eq:E_zero_and_Ric_zero}
\mathcal E(u)=0,
\qquad
\big[{\rm Ric}-(n-1)kg\big]
(\nabla u,\nabla u)=0
\quad\text{in }\Omega.
\end{equation}

Finally, define $v=u/f(0)$. Since $f(0)>0$, the function $v$ is positive in $\Omega$, vanishes on $\partial\Omega$, and satisfies
\begin{equation}
{\rm Hess}\,v
=
-\left(\frac1n+kv\right)g.
\end{equation}
By \cite[Lemma~6]{Farina:2022}, $\Omega$ is a metric ball and $v$ is radial.
The same conclusion holds for $u$.
\end{proof}

\begin{remark}
The equality conclusion in proposition \ref{prop:rigidity_equality} does not require $\mathcal E(u)\leq0$ as an initial hypothesis. Furthermore, if $\mathcal E(u)\leq0$ and equality holds in Corollary~\ref{cor:HK-integral-condition}, then \eqref{eq:thm1} implies
$\mathcal E(u)=0$ and equality in the refined inequality. Consequently, the same rigidity conclusion follows.
\end{remark}

To conclude the section, when $\mathcal{E}(u) \leq 0$, the next results extends the Theorem 2 of de Lima, Santos and Sindeaux in \cite{deLima:2026}.

\begin{theorem}[Soap Bubble rigidity]\label{thm:soap_truncated}
Let $M^n$ be a Riemannian manifold satisfying \eqref{Ricci}, with $k\in\mathbb R$, and let $\Omega\Subset M$ be a domain with $C^2$-boundary.
Let $u\in C^3(\Omega)\cap C^2(\overline{\Omega})$ be a positive solution of \eqref{P}, with $p=2$ and $f\in C^1(\mathbb R)$. Assume that $H>0$ on $\partial\Omega$ and $f(0)>0$. Define
\begin{equation*}
c
=
-\frac1{|\partial\Omega|}
\int_{\partial\Omega}u_\eta\,d\sigma
=
\frac1{|\partial\Omega|}
\int_\Omega f(u)d\Omega
>0,
\qquad
H_0=\frac{(n-1)f(0)}{nc}.
\end{equation*}
Then
\begin{equation}\label{eq:soap_error_truncated}
\int_{\partial\Omega}
(H_0-H)u_\eta^2d\sigma
\geq
H_0\int_{\partial\Omega}
(u_\eta+c)^2d\sigma
-
\frac{n-1}{n}\mathcal E(u)
\geq
-\frac{n-1}{n}\mathcal E(u).
\end{equation}

Moreover, if $H\geq H_0$ on $\partial\Omega,$ and $\mathcal E(u)\leq0,$
then
\begin{equation*}
\mathcal E(u)=0,
\qquad
H\equiv H_0,
\qquad
u_\eta\equiv-c
\quad\text{on }\partial\Omega,
\end{equation*}
and
\begin{equation}\label{eq:hessian_rigidity_truncated}
{\rm Hess}\,u
=
\frac{\Delta u}{n}g
=
-\frac{f(u)}n g
=
-\left(\frac{f(0)}n+ku\right)g
\quad\text{in }\Omega.
\end{equation}
In particular,
\begin{equation}\label{eq:f_affine_range_truncated}
f(s)=f(0)+nks
\qquad
\text{for every }
s\in[0,\max_{\overline{\Omega}}u].
\end{equation}
Consequently, $\Omega$ is a metric ball and $u$ is radial with respect to its center.
\end{theorem}

\begin{proof}
Since $f(0)>0$, continuity implies that $f(u)>0$ in a sufficiently small neighborhood of $\partial\Omega$. The boundary point lemma,
together with $u>0$ in $\Omega$ and $u=0$ on $\partial\Omega$, gives
$u_\eta<0$ on $\partial\Omega$. Hence $c>0$ and $H_0>0$. By \eqref{eq:key-identity},
\begin{equation}
\begin{split}
&\int_\Omega
|\mathring{{\rm Hess}}\,u|^2d\Omega
+
\int_\Omega
\big[{\rm Ric}-(n-1)kg\big]
(\nabla u,\nabla u)d\Omega
\\
&\qquad=
\frac{n-1}{n}f(0)c|\partial\Omega|
-
\int_{\partial\Omega}Hu_\eta^2d\sigma
+
\frac{n-1}{n}\mathcal E(u).
\end{split}
\end{equation}
On the other hand, the definition of $c$ gives
\begin{equation}
\int_{\partial\Omega}(u_\eta+c)^2d\sigma
=
\int_{\partial\Omega}u_\eta^2d\sigma
-
c^2|\partial\Omega|.
\end{equation}
Since $H_0c^2
=
\frac{n-1}{n}f(0)c,
$
combining these identities yields
\begin{equation}\label{eq:soap-deficit-identity}
\begin{split}
\int_{\partial\Omega}
(H_0-H)u_\eta^2d\sigma
&=
H_0\int_{\partial\Omega}
(u_\eta+c)^2d\sigma+
\int_\Omega
|\mathring{{\rm Hess}}\,u|^2d\Omega
\\
&\quad+
\int_\Omega
\big[{\rm Ric}-(n-1)kg\big]
(\nabla u,\nabla u)d\Omega-
\frac{n-1}{n}\mathcal E(u).
\end{split}
\end{equation}
The Ricci lower bound and $H_0>0$ now imply
\eqref{eq:soap_error_truncated}. Suppose additionally that $H\geq H_0$ and $\mathcal E(u)\leq0$. The left-hand side of \eqref{eq:soap-deficit-identity} is nonpositive, whereas every term on its right-hand side is nonnegative. Therefore, all these terms vanish. In particular,
\begin{equation}
\mathcal E(u)=0,
\qquad
u_\eta=-c
\quad\text{on }\partial\Omega,
\qquad 
\mathring{{\rm Hess}}\,u=0
\quad\text{in }\Omega,
\end{equation}
and
\begin{equation}
\big[{\rm Ric}-(n-1)kg\big]
(\nabla u,\nabla u)=0
\quad\text{in }\Omega.
\end{equation}
Furthermore,
\begin{equation}
0
=
\int_{\partial\Omega}
(H-H_0)u_\eta^2d\sigma
=
c^2\int_{\partial\Omega}(H-H_0)d\sigma.
\end{equation}              
Since $c>0$ and $H-H_0\geq0$, continuity implies $H\equiv H_0$ on $\partial\Omega$. Since
\begin{equation}
Hu_\eta=-H_0c=-\frac{n-1}{n}f(0)
\qquad\text{on }\partial\Omega,
\end{equation}
conditions \eqref{eq:tracefree-zero},\eqref{eq:ricci-equality-p2}, and \eqref{eq:boundary_relation} hold. By Corollary~\ref{cor:HK_equality}, equality holds in \eqref{eq:thm1}. Hence,  Proposition~\ref{prop:rigidity_equality} implies that \eqref{eq:hessian_rigidity_truncated},
\eqref{eq:f_affine_range_truncated}, and the claimed rigidity.
\end{proof}


\section{An Example Beyond the Global Derivative Bound}

In this section, we construct a smooth nonlinearity for which the global condition
\begin{equation}
f'(s)\leq nk
\qquad\text{for every }s\in\mathbb R
\end{equation}
fails, while the integral condition
\begin{equation}
\mathcal E(u)
=
\int_\Omega
\big(f'(u)-nk\big)|\nabla u|^2d\Omega
\leq0
\end{equation}
holds for a positive solution of \eqref{P}, with $p=2$. The construction illustrates that the behavior of $f$ outside the range of the solution need not satisfy the structural bound.

Let $(M^n,g)$ satisfy ${\rm Ric}\geq(n-1)kg$, with $k>0$, and let
$\Omega\Subset M$ be a domain with smooth boundary. Write $b=nk>0$. We first introduce the smooth nondecreasing cutoff
function
\begin{equation}\label{eq:example-cutoff}
\chi(t)
=
\begin{cases}
0, & t\leq0,\\[2mm]
\displaystyle
\frac{e^{-1/t}}
{e^{-1/t}+e^{-1/(1-t)}},
& 0<t<1,\\[4mm]
1, & t\geq1.
\end{cases}
\end{equation}
Thus, $\chi\in C^\infty(\mathbb R)$,
$0\leq\chi\leq1$, and $\chi'\geq0$. Fix constants $a,A,L,\delta,\omega>0$ such that
\begin{equation}
0<\delta<L,
\qquad
\delta\omega>b.
\end{equation}
Define
\begin{equation}\label{eq:example-nonlinearity}
\begin{split}
f(s)
&=
\big(1-\chi(s-A)\big)(a+bs)
\\
&\quad+
\chi(s-A)
\left(
-L
+
\delta\chi(s-A-1)
\sin\big(\omega(s-A-2)\big)
\right).
\end{split}
\end{equation}
This function belongs to $C^\infty(\mathbb R)$. For $s\leq A$, it satisfies $f(s)=a+bs,$ and, in particular, $f(0)=a>0$. On the transition interval $[A,A+1]$, the second cutoff vanishes, so
\begin{equation}
f(s)
=
\big(1-\chi(s-A)\big)(a+bs)
-
L\chi(s-A).
\end{equation}
Differentiating, we obtain
\begin{equation}
f'(s)
=
b\big(1-\chi(s-A)\big)
-
\chi'(s-A)(a+bs+L)
\leq b.
\end{equation}
Together with the linear expression on $[0,A]$, this gives
\begin{equation}\label{Eq0}
f'(s)\leq nk
\qquad\text{for every }s\in[0,A+1].
\end{equation}

For $s\geq A+1$, the first cutoff is equal to one. Consequently,
\begin{equation*}
f(s)
=
-L
+
\delta\chi(s-A-1)
\sin\big(\omega(s-A-2)\big)
\leq-L+\delta<0.
\end{equation*}
On the other hand, for $s\geq A+2$,
\begin{equation*}
f(s)
=
-L+\delta\sin\big(\omega(s-A-2)\big),
\end{equation*}
and therefore
\begin{equation*}
f'(s)
=
\delta\omega\cos\big(\omega(s-A-2)\big).
\end{equation*}
In particular, at
\begin{equation*}
s_j=A+2+\frac{2\pi j}{\omega},
\qquad j=0,1,2,\ldots,
\end{equation*}
we have $f'(s_j)=\delta\omega>nk.$ Thus, the global derivative bound fails. We now consider
\begin{equation}\label{eq:example_D}
\begin{cases}
-\Delta u=f(u), & \text{in }\Omega,\\
u=0, & \text{on }\partial\Omega.
\end{cases}
\end{equation}
The constant functions $\underline u=0$ and $\overline u=A+1$, form an ordered subsolution and supersolution. Indeed,
\begin{equation}
-\Delta\underline u=0<f(0)=a
\end{equation}
and
\begin{equation}
-\Delta\overline u=0>f(A+1)=-L,
\end{equation}
while
$\underline u=0\leq\overline u$ on $\partial\Omega$. For completeness, choose $\mu>0$ sufficiently large that the function
\begin{equation}
s\longmapsto f(s)+\mu s
\end{equation}
is nondecreasing on $[0,A+1]$. Starting with $u_0=0$, consider the iteration
\begin{equation*}
\begin{cases}
(-\Delta+\mu)u_{j+1}
=
f(u_j)+\mu u_j,
& \text{in }\Omega,\\
u_{j+1}=0,
& \text{on }\partial\Omega.
\end{cases}
\end{equation*}
The comparison principle gives
\begin{equation}
0\leq u_j\leq u_{j+1}\leq A+1.
\end{equation}
Standard elliptic estimates and compactness yield a solution of \eqref{eq:example_D} satisfying $0\leq u\leq A+1$. Since the domain, metric, and
nonlinearity are smooth, elliptic regularity gives $u\in C^\infty(\overline{\Omega})$. This solution is positive in $\Omega$.
Indeed, if $u(x_0)=0$ at an interior point, then $x_0$ is a minimum and $\Delta u(x_0)\geq0$, whereas the equation gives
\begin{equation*}
-\Delta u(x_0)=f(0)=a>0,
\end{equation*}
a contradiction.

Let $U=\max_{\overline{\Omega}}u.$ Since $u>0$ in $\Omega$ and vanishes on
$\partial\Omega$, its maximum is attained at an interior point $x_1$. If $U\geq A+1$, then $-\Delta u(x_1)\geq0,$ but
\begin{equation}
-\Delta u(x_1)=f(U)<0,
\end{equation}
which is impossible. Therefore,
\begin{equation}
0<u<A+1
\qquad\text{in }\Omega.
\end{equation}
By \eqref{Eq0}, it follows that $f'(u)\leq nk$ in $\Omega.$ Hence,
\begin{equation}
\mathcal E(u)
=
\int_\Omega
\big(f'(u)-nk\big)|\nabla u|^2d\Omega
\leq0.
\end{equation}
We have thus obtained a smooth nonlinearity violating the global derivative bound, together with a positive solution satisfying the integral
condition. Figure~\ref{fig:example-nonlinearity} illustrates the construction for
$$
a=b=A=1,
\qquad
L=2,
\qquad
\delta=\frac12,
\qquad
\omega=8.
$$

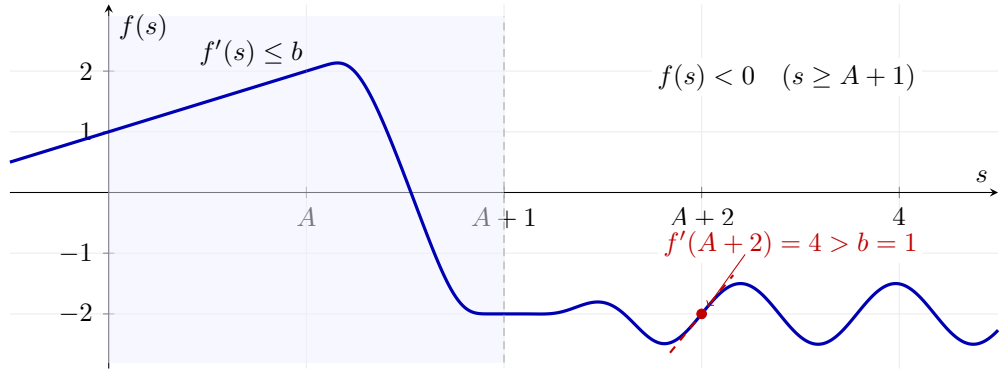
\begin{figure}[H]
\centering
\begin{tikzpicture}
\begin{axis}[
width=0.94\linewidth,
height=6.4cm,
axis lines=middle,
xmin=-0.5,
xmax=4.5,
ymin=-2.9,
ymax=3.1,
xlabel={$s$},
ylabel={$f(s)$},
xtick={0,1,2,3,4},
xticklabels={$0$,$A$,$A+1$,$A+2$,$4$},
ytick={-2,-1,0,1,2},
tick label style={font=\small},
label style={font=\small},
grid=major,
grid style={gray!15},
clip=false,
declare function={
    cut(\t)=
    exp(-1/max(\t,0.0001))/
    (
    exp(-1/max(\t,0.0001))
    +
    exp(-1/max(1-\t,0.0001))
    );
}
]

\path[fill=blue!6, fill opacity=0.5]
    (axis cs:0,-2.8)
    rectangle
    (axis cs:2,2.9);

\addplot[
    blue!70!black,
    very thick,
    domain=-0.5:1,
    samples=2
]
    {1+x};

\addplot[
    blue!70!black,
    very thick,
    domain=1:2,
    samples=180
]
    {(1-cut(x-1))*(1+x)-2*cut(x-1)};

\addplot[
    blue!70!black,
    very thick,
    domain=2:3,
    samples=180
]
    {-2+0.5*cut(x-2)*sin(deg(8*(x-3)))};

\addplot[
    blue!70!black,
    very thick,
    domain=3:4.5,
    samples=240
]
    {-2+0.5*sin(deg(8*(x-3)))};

\draw[densely dashed,gray!70]
    (axis cs:2,-2.8)
    --
    (axis cs:2,2.9);

\addplot[
    red!75!black,
    dashed,
    thick,
    domain=2.84:3.16,
    samples=2
]
    {-2+4*(x-3)};

\addplot[
    only marks,
    mark=*,
    mark size=1.8pt,
    red!75!black
]
    coordinates {(3,-2)};

\node[
    font=\small,
    inner sep=2pt
] at (axis cs:0.72,2.3)
    {$f'(s)\leq b$};

\node[
    font=\small,
    fill=white,
    inner sep=2pt
] at (axis cs:3.43,1.9)
    {$f(s)<0\quad(s\geq A+1)$};

\node[
    font=\small,
    text=red!75!black,
    fill=white,
    inner sep=2pt
] at (axis cs:3.45,-0.85)
    {$f'(A+2)=4>b=1$};

\draw[->,red!75!black]
    (axis cs:3.22,-1.02)
    --
    (axis cs:3.03,-1.88);

\end{axis}
\end{tikzpicture}
\caption{
The nonlinearity \eqref{eq:example-nonlinearity} for $a=b=A=1$, $L=2$, $\delta=1/2$, and $\omega=8$. The shaded interval contains the range of the positive solution and satisfies $f'\leq b=nk$. Outside this interval, the function remains negative but has portions with $f'>b$.
}
\label{fig:example-nonlinearity}
\end{figure}

\begin{remark}
This example distinguishes the integral condition from a derivative bound imposed on all of $\mathbb R$. It does not establish a strict distinction between $\mathcal E(u)\leq0$ and the pointwise condition restricted to $[0,\max_{\overline{\Omega}}u]$, because the latter condition also holds for the solution constructed here.
\end{remark}

\section*{Declarations}

\subsection*{Data availability statement} 
This manuscript has no associated data.

\subsection*{Competing interests} 
The authors have no competing interests to declare that are relevant to the content of this article.

\subsection*{Authors contributions}
All authors developed the theoretical formalism, performed the analytic calculations, discussed the results and contributed to write the manuscript.

\subsection*{Funding}
The first author is partially supported by CNPq, Brazil,
grant $\text{PDJ}-153013/2025-7$ and FACEPE, Brazil, grant BFP$-0015-1.01/25$. The second author is partially supported by CNPq, Brazil, Universal Project grant number 406078/2025-4.

\bibliographystyle{plain}

\end{document}